\documentclass{birkjour}
 \newtheorem{thm}{Theorem}[section]
 
 \newtheorem{lem}[thm]{Lemma}
 
 \theoremstyle{definition}
 \newtheorem{defn}[thm]{Definition}
 \theoremstyle{remark}
 \newtheorem{rem}[thm]{Remark}
 
 \usepackage[mathscr]{eucal}
 \usepackage[OT1,T1,TS1,T2A]{fontenc}
\usepackage[cp1251]{inputenc}
\usepackage[russian,english]{babel}
 \numberwithin{equation}{section}
\newcommand{\ccomma}{\mathpunct{\raisebox{0.5ex}{,}}}

\begin{document}

%
%
%
%
%

\title[A positivity property suggested by BMV ]
 {A positivity property\\ of a Quantum Anharmonic Oscillator\\
  suggested by the BMV conjecture}
\author[V. Katsnelson]{Victor Katsnelson}
\address{%
Department of Mathematics\\
The Weizmann Institute\\
76100, Rehovot\\
Israel}
\email{victor.katsnelson@weizmann.ac.il; victorkatsnelson@gmail.com}
\date{October 1, 2014}

\dedicatory{Dedicated to the blessed memory of Herbert Stahl}
\begin{abstract}
In this work an observation concerning a positivity property
of the quantum anharmonic oscillator is made. This positivity
property is suggested by the BMV conjecture.
\end{abstract}
\maketitle
 \section{Herbert Stahl's Theorem.}
In the paper \cite{BMV} a conjecture was
formulated which now is commonly known as the BMV conjecture:\\[1.0ex]
\textbf{The BMV Conjecture.} Let \(A\) and \(B\) be Hermitian matrices of size
\(n\times{}n\). Assume moreover that the matrix \(B\not=0\) is positive semidefinite that is
\begin{equation}
\label{trf}
x^{\ast}Bx\geq0 \ \ \forall\ \ n\times1\ \textup{vector-columns}\ x.
\end{equation}
Then the function
\begin{equation}
\label{TrF}
\varphi(t)=
\textup{trace}\,\{\exp[-(A+tB)]\}
\end{equation}
of the variable \(t\) is representabe as a Laplace transform of a \textsf{non-negative} measure
\(d\sigma_{A,B}(\lambda)\) supported on the positive half-axis:
\begin{equation}
\label{LaR}
\varphi(t)=\!\!\int\limits_{\lambda\in[0,\infty)}\!\!\exp(-\lambda{}t)\,d\sigma_{A,B}(\lambda), \ \ \forall \,t\in(0,\infty).
\end{equation}

{\ }\\[3.0ex]
Let us note that the function \(\varphi(t)\), considered for \(t\in\mathbb{C}\), is an entire function of exponential type. The indicator diagram of the function \(\varphi\) is
the closed interval \([-\lambda_{\max},-\lambda_{\min}]\), where \(\lambda_{\min}\) and
\(\lambda_{\max}\) are the least and the greatest eigenvalues of the matrix \(B\) respectively.
Thus if the function \(\varphi(t)\) is representable in the form \eqref{LaR} with a non-negative
measure \(d\sigma_{A,B}(\lambda)\), then \(d\sigma_{A,B}(\lambda)\) is actually supported on the interval
\([\lambda_{\min},\lambda_{\max}]\) and the representation
\begin{equation}
\label{LaRm}
\varphi(t)=
\hspace*{-8.0pt}\int\limits_{\lambda\in[\lambda_{\min},\lambda_{\max}]}\hspace*{-8.0pt}
\exp(-\lambda{}t)\,d\sigma_{A,B}(\lambda), \ \ \forall \,t\in\mathbb{C},
\end{equation}
holds for every \(t\in\mathbb{C}\).

The representability of the function \(\varphi(t)\), \eqref{TrF}, in the form \eqref{LaRm}
with a non-negative \(d\sigma_{A,B}\) is evident if the matrices \(A\) and \(B\) commute.
In this case \(d\sigma(\lambda)\) is an atomic measure located on the spectrum of the matrix \(B\).
In general case, if the matrices \(A\) and \(B\) do not commute, the BMV conjecture remained an
open question for longer than 35 years. In 2011, Herbert
Stahl gave an affirmative answer to the BMV conjecture.\\[1.0ex]
\textbf{Theorem}\,(H.Stahl) \textit{Let \(A\) and \(B\) are \(n\times{}n\) hermitian matrices and \(B\) is positive semi-definite.}
\textit{Then the function \(\varphi(t)\) defined by \eqref{TrF} is representable
as the Laplace transform \eqref{LaRm} of a non-negative measure \(d\sigma_{A,B}(\lambda)\)
supported on the interval \([\lambda_{\min},\lambda_{\max}]\).}

The first arXiv version of H.Stahl's Theorem appeared in \cite{S1}, the latest arXiv version -
in \cite{S2}, the journal publication - in \cite{S3}.

The proof of Herbert Stahl is based on ingenious considerations related to
Riemann surfaces of algebraic functions. In \cite{E}, a simplified version of the Herbert Stahl proof is presented.
\section{The goal of this paper.}
In the BMV conjecure, which now is the Herbert Stahl theorem, the subject of consideration is the function \(\varphi(t)\) of the form \eqref{TrF}, where
a linear pencil \(A+tB\) of square matrices \(A\) and \(B\) of arbitrary \textsf{finite} size
appears in the exponential. The goal of the present paper is to discuss a~special example
of a function \(\varphi(t)=\text{trace}\{\exp[-(A+tB)]\}\), where \(A+tB\) is
a linear pencil of unbounded non-negative operators in the space \(L^2(\mathbb{R})\).

Namely we consider the operators \(A\) and \(B\) generated in \(L^2(\mathbb{R})\) by the expressions
\begin{gather}
\label{DO}
(Af)(x)=-\frac{d^2f(x)}{dx^2},\\
\label{MO}
(Bf)(x)=V(x)f(x),
\end{gather}
where the following conditions are posed on the real-valued function \(V\):
\begin{enumerate}
\item
The function \(V\) is defined and continuous on the real axis: \(V\in{}C(-\infty,\infty)\).
\item {\ }
\vspace{-3.5ex}
\begin{gather}
\label{mon}
V(0)=0, \\
 \hspace{2.0ex} V(x) \ \text{is strictly increasing on }  [0,+\infty):\,V(x_1)<V(x_2) \text{ if } 0\leq{}x_1<x_2, \notag\\
\hspace{2.0ex} V(x) \ \text{is strictly decreasing on }  (-\infty,0]:\,V(x_1)>V(x_2) \text{ if } x_1<x_2\leq0.\notag
\end{gather}
{\ }\\[-3.0ex]
In particular, \(V(x)>0\) for \(x\in(-\infty,\infty)\setminus 0\).\\
\item{\ }
\vspace{-3.9ex}
\begin{equation}
\label{Unb}
V(x)\to+\infty \ \text{ as } \ x\to\pm\infty\,.
\end{equation}
\end{enumerate}

Let \(\mathscr{D}\) be the set of all smooth complex valued compactly supported functions,
\(\mathscr{D}\subset{}L^2(\mathbb{R})\).
Each of the operators \(A\) and
\(B\) is a symmetric operator defined on \(\mathscr{D}\). Both of these operators are non-negative on
\(\mathscr{D}\):
\begin{equation}
\label{Nn}
\langle{}Af,f\rangle\geq0,\quad \langle{}Bf,f\rangle\geq0, \ \ \forall f\in\mathscr{D},
\end{equation}
where \(\langle\,.\,,\,.\,\rangle\) is the standard scalar product in \(L^2(\mathbb{R})\).
For each \(t>0\), the linear combination \((A+tB)f\) is defined for every \(f\) from
common domain of definition \(\mathscr{D}\) of the operators \(A\) and \(B\). Moreover, the operator \(L_t=(A+tB)\) is non-negative on \(\mathscr{D}\):
\begin{equation}
\label{NonN}
\big\langle{}(A+tB)f,f\big\rangle\geq0, \ \ \forall \ \ f\in\mathscr{D}.
\end{equation}
The operator \(L_t=A+tB\) admits a selfadjoint
extension from \(\mathscr{D}\) to a domain of definition \(\mathscr{D}_t\).
This selfadjoint extension is non-negative on \(\mathscr{D}_t\). We preserve the notation \(L_t\) for the extended operator. It turns out that such extension is unique.

\emph{The operator \(L_t=A+tB\), where \(A\) is of the form \eqref{DO},  \(B\) is of the form
\eqref{MO}, the function \(V\) satisfies the conditions \eqref{mon} and \eqref{Unb}, and \(t>0\)
is called \textsf{the quantum anharmonic oscillator}. The function \(V(x)\) is said to be
\textsf{the potential function}\\
In the special case \(V(x)=x^2\), the operator \(L_t\) is said to be
\textit{\textsf{the quantum harmonic oscillator}}. }\\

For one-parametric family \(L_t=-\frac{d^2}{dx^2}+tV(x)\) of anharmonic oscillators,
we formulate the conjecture which is an analog the BMV conjecture. For one-parametric family \(L_t\) of anharmonic oscillators with a \emph{homogeneous} potential
function (see Definition below), we confirm the analog of the BMV conjecture even in a stronger form.\\

\noindent
\begin{defn}
\label{DeHom} \textit{Let \(\rho>0\) be a positive number. The potential function \(V\)
is said to be \textsf{homogeneous of order \(\rho\)} if}
\begin{equation}
\label{hom}
V(\xi x)=\xi^{\rho}V(x), \ \ \forall\,\xi>0,\,\forall\,x\in(-\infty,\infty).
\end{equation}
\end{defn}

\vspace{2.0ex}
It is clear that any homogeneous potential function \(V\) is of the form
\begin{gather}
\label{Pot}
V(x)=
\begin{cases}
c^{+}|x|^{\rho}, \ \text{for} \ x>0,\\
c^{-}|x|^{\rho}, \ \text{for} \ x<0,
\end{cases}
\end{gather}
where
\begin{equation}
\rho>0,\,c^{+}>0,\,c^{-}>0 \ \text{are strictly positive constants}.
\end{equation}

{\ }\\
\section{The spectrum of quantum anharmonic oscillator.}
We would not like to discuss the questions related to the domain of definition of the
operator~\(L_t\). The only fact which is important for us is the following:\\

\noindent
\textbf{Lemma 1.}
\begin{enumerate}
\item[\textbf{1.}]
\textit{For each \(t\in(0,\infty)\), the spectrum of the anharmonic operator \(L_t\) is discrete.}
\textit{The eigenvalue problem
\begin{equation}
\label{EVPt}
-\frac{d^2f(x)}{dx^2}+tV(x)f(x)=\lambda(t)f(x), \ \ f(x)\in{}L^2(\mathbb{R}),\ f\not\equiv0,
\end{equation}
has solution only for \(\lambda(t)=\lambda_n(t)\), where \(\lambda_n(t)\),
\(n=0,\,1,\,2,\,3,\,\ldots\),
is a~sequence of strictly positive numbers tending to \(\infty\)}:
\begin{gather}
\label{SpSeT}
0<\lambda_0(t)<\lambda_1(t)<\lambda_2(t)<\lambda_3(t)<\,\ldots,\\
\label{SpSeI}
\lambda_n(t)\to\infty \ \text{as} \ n\to\infty, \ \forall\,t>0.
\end{gather}
\item[\textbf{2.}]
\textit{If the potential function \(V\) of the anharmonic oscillator is homogeneous of order
\(\rho>0\), then
for each \(n=0,\,1,\,2,\,3,\,\ldots\), the eigenvalue \(\lambda_n(t)\) is a homogeneous function of \(t\)
of order \(2/(2+\rho)\)}:
\begin{equation}
\label{HoOr}
\lambda_n(t)=t^{2/(2+\rho)}\lambda_n(1),\ \ n=0,\,1,\,2,\,3,\,\ldots , \ 0<t<\infty.
\end{equation}
\item[\textbf{3.}]
\textit{For each \(t>0\), the operator \(e^{-L(t)}\) is the trace class operator}:
\begin{equation}
\label{TrC}
\textup{trace}\,e^{-L(t)}=\sum\limits_{0\leq{}n<\infty}\exp[-\lambda_n(t)]<\infty,\
\ \forall\,t\in(0,\infty).
\end{equation}
\end{enumerate}
\begin{proof} {\ }\\
\textbf{1.} For each \(t>0\), the operator \(L_t\) is of the form
\begin{equation}
\label{StL}
-\frac{d^2{\ }}{dx^2}+V(x),
\end{equation}
where the potential function \(V(x)\geq 0\) satisfies the condition \eqref{Unb}.
It is well known\footnote{See for example \cite[Chapt.II, Sect.28]{Gl}, Theorem 5.} that under condition \eqref{Unb}, the spectrum of the operator \eqref{StL}
is discrete. Since \(V(x)\) is positive for \(x\not=0\), the spectrum is strictly positive. Thus the conditions \eqref{SpSeT} and
\eqref{SpSeI} hold.\\
\textbf{2.} Let \(\lambda\) be an eigenvalue of the operator \(L_1\), that is the equation
\begin{equation}
\label{EVP1}
-\frac{d^2f(x)}{dx^2}+V(x)f(x)=\lambda{}f(x), \ \ f(x)\in{}L^2(\mathbb{R}),\ f\not\equiv0,
\end{equation}
has a solution. Taking arbitrary \(\xi>0\), we change variable \(x\to{}y=\xi{}x\). The equation \eqref{EVP1}
will be transformed to the equation
\begin{equation}
\label{EVPtCh}
-\frac{d^2g(y)}{dy^2}+tV(y)g(x)=\lambda(t)g(y),
\end{equation}
where
\begin{equation}
\label{ChT}
t=\xi^{-(\rho+2)},\ \lambda(t)=t^{2/\rho+2}\lambda, \ g(y)=f(t^{1/(\rho+2)}y).
\end{equation}

\vspace*{2.0ex}
\noindent
\textbf{3.}
\vspace*{-2.8ex}
\begin{equation}
\text{For } r>0, \ \text{let } \ N(r)=\#\{k:\,\lambda_k< r\},
\end{equation}
-- the number of those eigenvalues \(\lambda_k\) of the eigenvalue problem \eqref{EVP1}
which are less than \(r\).
In \cite{deWM}, the asymptotic relation
\begin{equation}
\label{GAR}
N(r)\sim\int\limits_{\eta:V(\eta)<r}\sqrt{r-V(\eta)}\,d\eta, \ \ r\to+\infty.
\end{equation}
was established for a wide class of potential functions \(V(x)\). In particular, \eqref{GAR}
holds for any homogeneous potential function \(V\).
Using \eqref{Pot}, one can transform \eqref{GAR} to the form
\begin{equation}
\label{AsR}
N(r)\sim C_{V}\cdot{}r^{\frac{1}{2}+\frac{1}{\rho}} \ \ \text{as} \ \ r\to+\infty,
\end{equation}
 where
\begin{equation}
\label{CV}
C_V=(c_{+}^{1/\rho}+c_{-}^{1/\rho})\frac{1}{\rho}\frac{1}{2\sqrt{\pi}}
\frac{\Gamma(\frac{1}{\rho})}{\Gamma(\frac{1}{\rho}+\frac{3}{2})}\cdot
\end{equation}
The condition \eqref{TrC} follows from the asymptotic relation \eqref{AsR} and from
\eqref{HoOr}. Indeed
\begin{multline*}
\sum\limits_{0\leq n<\infty}e^{-\lambda_n(t)}=
\sum\limits_{0\leq n<\infty}e^{-t^{\frac{2}{\rho+2}}\lambda_n}=
\int\limits_{0\leq{}r<\infty}e^{-t^{\frac{2}{\rho+2}}r}\,dN(r)=\\
=t^{\frac{2}{\rho+2}}\int\limits_{0\leq{}r<\infty}e^{-t^{\frac{2}{\rho+2}}r}\,N(r)\,dr<\infty\,.
\end{multline*}
\end{proof}

\vspace{2.0ex}
\noindent
\begin{rem}
The detailed spectral analysis of the quantum harmonic oscillator was done by P.A.M.\,Dirac in 1930, \cite{D}.
For quantum harmonic oscillator \(L=-\dfrac{d^2\,}{dx^2}+x^2\), the eigenvalues \(\lambda_n\)
and the eigenfunctions \(h_n(x)\) are:
\begin{equation}
\label{SpHO}
\lambda_n=2n+1,\ \ h_n(x)=e^{x^2/2}\cdot\frac{d^n}{dx^n}e^{-x^2}\,.
\end{equation}
The method of spectral analysis invented by Dirac is purely algebraic. Now this method is known as \emph{the method of ladder operators}.
For systematic presentation of the method of ladder operators we refer to \cite[sec.2.3.1]{G}.
\end{rem}
\section{Absolutely monotonic functions.}
\noindent
\begin{defn}\label{DAMF}
 Let \(\Phi(t)\) be a function defined on an interval
\((a,b),\,(a,b)\subset\mathbb{R}\).
 The function \(\Phi(t)\) is said to be \textit{absolutely
monotonic on \((a,b)\)} if it satisfies the conditions
\begin{multline}
\label{PC}
\hfill\Delta_h^n\Phi(t)\geq0, \ \ \forall\,n=0,\,1,\,2,\,3,\,\ldots,
\forall\,t,\,h\,\in\mathbb{R}:\,a<t+nh<b,\hfill
\end{multline}
where \(\Delta_h^n\Phi(t)\stackrel{\text{\tiny def}}{=}
\sum\limits_{k=0}^{n}(-1)^{n-k}\binom{n}{k}\Phi(t+kh)\)
is the \(n\)-th difference of the function~\(\Phi\).
\end{defn}

\vspace{2.0ex}
In implicit form, absolutely monotonic functions appeared in the paper \cite{Ber1}
of S.N.\,Berstein. (The terminology "absolutely monotomic function" did not appear in \cite{Ber1}.) Systematic presentation of the theory of absolutely monotonic functions was done
in the paper \cite{Ber2}.
\noindent
Concerning absolutely monotonic functions and related questions, we refer to the books of
N.I.\,Akhiezer \cite[Chapter V, Section 5]{A} and D.W.\,Widder \cite[Chapter 4]{Wid}.

\vspace{3.0ex}
\noindent
\textbf{Theorem} (S.N.\,Bernstein)\\
\textbf{1.} Let a function \(\Phi(t)\) be absolutely monotonic on some interval \((a,b)\) of
the real axis. Then the function \(\Phi\) is infinitely differentiable on \((a,b)\):
\(\Phi\in{}C^{\infty}(a,b)\), and the conditions
\begin{equation}
\label{PD}
\Phi^{(n)}(t)\geq0, \ \ \forall\,t\in\,(a,b),\,n=0,\,1,\,2,\,\ldots\,,
\end{equation}
are satified, where \(\Phi^{(n)}(t)\) is the \(n\)-th derivative of the function \(\Phi\).\\
\textbf{2.} Let a function \(\Phi(t)\) be infinitely differentiable on some interval \((a,b)\) of the real axis, and let the conditions \eqref{PD} hold. Then the function \(\Phi\) is
absolutely monotonic on the interval \((a,b)\), that is the conditions \eqref{PC} hold.

\vspace{3.0ex}
\noindent
\begin{lem}\label{Lem 0}
 \textit{Let \(\{\Phi_m(t)\}\) be a sequence of functions each of them is
absolutely monotonic on an interval \((a,b)\) of the real axis. Assume that for each
\(t\in(a,b)\) there exist the final limit
\begin{equation}
\label{FL}
\Phi(t)\stackrel{\text{\tiny def}}{=}\lim_{m\to\infty}\Phi_{m}(t), \ \ t\in(a,b).
\end{equation}
Then the limiting function \(\Phi\) is absolutely monotonic on the interval \((a,b)\).}
\end{lem}
\begin{proof}
Take and fix \(n\). Take arbitrary \(t\) and \(h\) so that \(t\in(a,b),\,t+nh\in(a,b)\).
Passing to the limit in the inequality \(\Delta_h^n\Phi_m(t)\geq0\) as \(m\to\infty\),
we come to the inequality \(\Delta_h^n\Phi(t)\geq0\).
\end{proof}
\vspace{3.0ex}
\noindent
\textbf{Theorem} (S.N.\,Bernstein - D.V.\,Widder).\\
\textbf{1.} Let \(\Phi(t)\) be a absolutely monotonic function on the negative half-axis
\(t:\{-\infty<t<0\}\). Then
there exists a non-negative measure \(d\sigma(\lambda)\) supported on the half-axis
\([0,\infty)\) such that the function \(\Phi\) is representable
in the form
\begin{equation}
\label{InRe}
\Phi(t)=\int\limits_{\lambda\in[0,\infty)}e^{t\lambda}\,d\sigma(\lambda), \ \
\forall \, t\in(-\infty,0).
\end{equation}
Such a mesure \(d\sigma(\lambda)\) is unique.\\
\textbf{2.} Let \(d\sigma(\lambda)\) be a non-negative measure supported on the half-axis
\([0,\infty)\). Assume that the integral in the right hand side of \eqref{InRe} is finite
for each \(t\in(-\infty,0)\). Then the function \(\Phi(t)\) which is \emph{defined} by the
equality \eqref{InRe} is absolutely monotonic on \((-\infty,0)\).\\

Thus the Herbert Stahl theorem can be reformulated as follows:\\
\textit{Let \(A\) and \(B\) be Hermitian \(n\times{}n\) matrices, and moreover \(B\geq0\).
Let the function \(\varphi(t)\) is defined by \eqref{TrF}
for \(t\in(0,\infty)\). Then the function \(\varphi(-t)\), considered as a function of the variable \(t\), is absolutely monotonic on \((-\infty,0)\).}\\

\noindent
\begin{lem}
\label{Lem 2}
 \textit{Let \(\Psi(t)\) be an infinitely differentiable real-valued function
defined on the half-axis
\((-\infty,0)\). Assume that the derivative \(\Psi^{\prime}(t)\) of the function \(\Psi\)
is an absolutely monotonic function on \((-\infty,0)\), that is
\begin{equation}
\label{DCM}
\Psi^{(k)}(t)\geq 0 \ \ \forall\,t\in(-\infty,0), \ \ k=1,\,2,\,3,\,\ldots\,.
\end{equation}
Then the function
\begin{equation}
\label{EF}
\Phi(t)\stackrel{\textup{\tiny def}}{=}\exp\,[\Psi(t)]
\end{equation}
is absolutely monotonic.}
\end{lem}
\noindent
\textit{Proof.} \
Since the function \(\Psi\) is real valued, the inequality
\begin{equation}
\label{Pos}
\Phi(t)>0, \ \ \forall \,t\in(-\infty,0),
\end{equation}
holds for the function \(\Phi\). To establish the inequalities
\begin{equation}
\label{kgo}
\Phi^{(k)}(t)\geq0, \ \ \forall\,t\in(-\infty,0), \ \ k=1,\,2,\,3,\,\ldots
\end{equation}
we remark that
\begin{equation}
\label{DP}
\Phi^{(k)}(t)=\Phi(t)
\cdot\,P_{k}(\Psi^{\prime}(t),\Psi^{\prime\prime}(t)\,\ldots,\,\Psi^{(k)}(t)),
\ \ k=1,2,\,\ldots\,,
\end{equation}
where \(P_k(y_1,\,\ldots\,,\,y_k)\) is a polynomial of the variables \(y_1,\,\ldots\,,y_k\) with \emph{non-negative} coefficients.
Indeed,
\begin{multline*}
P_1(y_1)=y_1, \ \ P_{k+1}(y_1,y_2,\,\ldots\,y_{k},y_{k+1})=\\
=y_1P_k(y_1,\,\ldots\,,\,y_k)+\sum\limits_{1\leq{}j\leq{k}}
\frac{\partial{}P_{k}}{\partial{}y_{j}}(y_1,\,\ldots\,,\,y_k)y_{j+1}, \ k=1,\,2,\,3,\,\ldots\,.
\tag*{\qed}
\end{multline*}

\begin{lem}\label{Lem 3} \textit{Given numbers \(a>0\) and \(\alpha\in(0,1)\), let
the function \(\Psi(t)\) be defined  as
\begin{equation}
\label{PF}
\Psi(t)=-a(-t)^{\alpha}, \ \ t\in(-\infty,0),
\end{equation}
where the branch of the function \(s^{\alpha}\) is chosen which takes
positive values for \(s>0\).}

\textit{Then the derivative \(\Psi^{\prime}(t)\) of the function \(\Psi(t)\) is a absolutely monotonic function on \((-\infty,0)\).}
\end{lem}
\begin{proof} \ For \(k=1,\,2,\,3,\,\ldots, \),
\begin{equation*}
\frac{d^k}{dt^k}\Psi(t)=
(-1)^{k-1}\alpha(\alpha-1)\cdot\,\,\cdots\,\,\cdot\big(\alpha-(k-1)\big)
(-t)^{\alpha-k}.
\end{equation*}
It is clear that \((-1)^{k-1}\alpha(\alpha-1)\cdot\,\,\cdots\,\,\cdot\big(\alpha-(k-1)\big)>0\).
\end{proof}

\begin{lem}\label{Lem 4} \textit{Given a numbers \(a>0\) and \(\alpha\in(0,1)\), let
the function \(\Phi_{a,\alpha}(t)\) be defined  as
\begin{equation}
\label{MF}
\Phi_{a,\alpha}(t)=\exp[-a(-t)^{\alpha}], \ \ t\in(-\infty,0),
\end{equation}
where the branch of the function \(s^{\alpha}\) is chosen which takes
positive values for \(s>0\).}

\textit{Then the function \(\Phi_{a,\alpha}\) is absolutely monotonic on the half-axis \((-\infty,0)\).}
\end{lem}
\begin{proof}
Lemma \ref{Lem 4} is a consequence of Lemmas \ref{Lem 2} and \ref{Lem 3}.
\end{proof}

\begin{lem}\label{Lem 5} \textit{Let \(\{a_n\}\) be a sequence of positive numbers and \(\alpha>0\).
Assume that}
\begin{equation}
\label{CoC}
\sum\limits_{n}e^{-a_n\tau}<\infty \ \ \forall\,\tau>0.
\end{equation}

\textit{Then the function
\begin{equation}
\label{TrE}
\Phi(t)=\sum\limits_{n}e^{-a_n(-t)^{\alpha}}
\end{equation}
is absolutely monotonic on \((-\infty,0)\).}
\end{lem}
\begin{proof}
Lemma 5 is a consequence of Lemma \ref{Lem 4} and Lemma \ref{Lem 0}.
\end{proof}
\noindent
\section{Main Theorem.}
\noindent
The following fact is a direct consequence of above stated reasonings.
\begin{thm}\label{MaTh}
\textit{Let
\begin{equation}
\label{OPF}
L_t=-\frac{d^2}{dx^2}+tV(x),
\end{equation}
where \(V(x)\) is of the form \eqref{Pot}, be an one-parametric family of quantum
anharmonic oscillators with a \emph{homogeneous} potential.}

\textit{Then the function
\begin{equation}
\label{TrEF}
\varphi(t)=\textup{trace}\,e^{-L_t}=\sum\limits_{0\leq{}n<\infty}e^{-\lambda_n(t)},
\end{equation}
where \(\{\lambda_n(t)\}_{0\leq{}n<\infty}\) is the sequence of all eigenvalues of the operator \(L_t\), is representable in the
form
\begin{equation}
\label{LaRe}
\varphi(t)=\int\limits_{\lambda\in[0,\infty)}e^{-\lambda{}t}\,d\sigma(\lambda), \ \
0<t<\infty,
\end{equation}
where \(d\sigma\) is a non-negative measure supported on \([0,\infty)\).}
\end{thm}

Actually we proved more. Namely, we proved that under assumptions of the above Theorem,
\textit{each summand} \(e^{-\lambda_n(t)}\) of the sum in the right hand side of \eqref{TrEF}
is representable in the form
\begin{equation}
\label{EaS}
e^{-\lambda_n(t)}=\int\limits_{\lambda\in[0,\infty)}e^{-\lambda{}t}\,d\sigma_n(\lambda), \ \
0<t<\infty,
\end{equation}
where \(d\sigma_n\) is a non-negative measure supported on \([0,\infty)\). So
\begin{equation}
\label{MS}
d\sigma=\sum\limits_{0\leq{}n<\infty}d\sigma_n,
\end{equation}
where \(d\sigma\) and \(d\sigma_n\) are the measures which appear in the integral representations \eqref{LaRe} and \eqref{EaS} respectively.

If the potential \(V\) is a homogeneous even function: \(V(x)=V(-x)\), that is if
\(c^{+}=c^{-}\) in \eqref{Pot}, then for each \(t\) the subspaces \(L^2_{\textup{ev}}(\mathbb{R})\)
and \(L^2_{\textup{od}}(\mathbb{R})\) of even and odd functions from \(L^2(\mathbb{R})\) are invariant with respect to the operator \(L_t\), \eqref{OPF}. In particular,
each eigenfunction of the operator \(L_t\) is either even or odd.
So the function \(\varphi(t)\), \eqref{TrEF}, splits into the sum of two functions
\begin{equation}
\label{Spl}
\varphi(t)=\varphi_{\textup{ev}}(t)+\varphi_{\textup{od}}(t),
\end{equation}
where the sums
\begin{equation}
\label{EOS}
\varphi_{\textup{ev}}(t)=\sideset{}{_{\textup{ev}}}\sum\limits_{n}e^{-\lambda_n(t)},\quad
\varphi_{\textup{od}}(t)=\sideset{}{_{\textup{od}}}\sum\limits_{n}e^{-\lambda_n(t)}
\end{equation}
are taken over \(n\) corresponding to even and odd eigenfunctions respectively.
From \eqref{EaS}, the integral representations
\begin{equation}
\varphi_{\textup{ev}}(t)=\int\limits_{\lambda\in[0,\infty)}e^{-\lambda{}t}\,
d\sigma_{\textup{ev}}(\lambda),
\ \
\varphi_{\textup{od}}(t)=\int\limits_{\lambda\in[0,\infty)}e^{-\lambda{}t}\,
d\sigma_{\textup{od}}(\lambda), \ \ 0<t<\infty,
\end{equation}
follow, where \(d\sigma_{\textup{ev}}\) and \(d\sigma_{\textup{od}}\) are non-negative measures
supported on \([0,\infty)\).\\[2.0ex]

In \, the \, case \, of \, the one-parametric \, family \, of \, harmonic \, oscillators,\\
\(L_t=-\dfrac{d^2}{dx^2}+tx^2\), the functions \(\varphi\), \(\varphi_{\textup{ev}}\), \(\varphi_{\textup{od}}\) can be calculated explicitly:
\begin{equation}
\label{ExE}
\varphi(t)=\frac{1}{\sh\sqrt{t}}\ccomma \quad
\varphi_{\textup{ev}}(t)=\frac{e^{\sqrt{t}}}{\sh\sqrt{2t}}\ccomma \quad
\varphi_{\textup{od}}(t)=\frac{e^{-\sqrt{t}}}{\sh\sqrt{2t}}\ccomma \quad 0<t<\infty.
\end{equation}
(See \eqref{SpHO} and \eqref{HoOr} with \(\rho=2\).) Without using Lemma 4 and the
representation \eqref{EOS}, it is not
so evident that the functions \(\varphi_{\textup{ev}}(-t)\), \(\varphi_{\textup{od}}(-t)\)
are absolutely monotonic on \((-\infty,0)\).

\section{A conjecture for the anharmonic oscillator\\ which is analogous to the BMV conjecture for matrices.}
\noindent
Let us formulate a conjecture. Assume that two functions \(V_0\) and \(V_1\) are given
which satisfy the conditions:
\begin{enumerate}
\item \(V_0\) and \(V_1\) are defined on the whole real axis \(\mathbb{R}\) and are continuous
there: \(V_0\in{}C(\mathbb{R})\), \ \(V_1\in{}C(\mathbb{R})\).
\item The positivity condition:
\begin{equation}
\label{PoC}
V_0(x)>0,\,V_1(x)>0 \quad\forall\,x\in(-\infty,\infty)\setminus0.
\end{equation}
\item The unboundedness condition
\begin{equation}
\label{UbC}
V_0(x)+V_1(x)\to\infty \ \ \textup{as} \ \ x\to\pm\infty.
\end{equation}
\end{enumerate}
{\ }

\vspace{0.5ex}
\noindent
Let as consider the one-parametric family of operators
\begin{equation}
\label{GOPF}
\mathscr{L}_t=-\frac{d^2}{dx^2}+\big(V_0(x)+tV_1(x)\big),
\end{equation}
where \(t\in(0,\infty)\).

The conditions \eqref{PoC} and \eqref{UbC} ensure that for each \(t>0\) the spectrum
of the operator \(\mathscr{L}_t\) is discrete and strictly positive. This means
that for each \(t>0\), the eigenvalue problem
\begin{equation}
\label{GEVPr}
-\frac{d^2f(x)}{dx^2}+\big(V_0(x)+tV_1(x)\big)f(x)=\lambda{}f(x), \ \ f(x)\in{}L^2(\mathbb{R}),\ f\not\equiv0
\end{equation}
is solvable only for those values \(\lambda=\lambda_n(t)\) which form a sequence tending to
\(+\infty\):
\begin{equation*}
0<\lambda_0(t)<\lambda_1(t)<\lambda_2(t)<\,\cdots\,, \ \
\lambda_n(t)\to+\infty \ \textup{as} \ n\to\infty.
\end{equation*}
Let us define the function \(\Phi(t)=\textup{trace}\,e^{-\mathscr{L}(t)}\) as
\begin{equation}
\label{GeTr}
\Phi(t)=\sum\limits_{n}e^{-\lambda_n(t)},
\end{equation}
where the sum is taken over all eigenvalues \(\lambda_n(t)\) of the eigenvalue
problem~\eqref{GEVPr}.\\

If both functions \(V_0\) and \(V_1\) are even: \(V_0(x)=V_0(-x)\), \(V_1(x)=V_1(-x)\),
then each eigenfunction \(f(x)\) of the eigenvalue problem \eqref{GEVPr} is either even or odd.
In this case we can define the "partial" traces
\begin{equation}
\label{PaTr}
\Phi_{\textup{ev}}(t)=\sideset{}{_{\textup{ev}}}\sum\limits_{n}e^{-\lambda_n(t)},\quad
\Phi_{\textup{od}}(t)=\sideset{}{_{\textup{od}}}\sum\limits_{n}e^{-\lambda_n(t)},
\end{equation}
where the sums \(\sideset{}{_{\textup{ev}}}\sum\) and \(\sideset{}{_{\textup{od}}}\sum\)
are taken over \(n\) corresponding even and odd eigenfunctions of the eigenvalue problem \eqref{GEVPr}.

Of course we should pose some condition on the functions \(V_0\) and \(V_1\) which ensure that
\begin{equation}
\label{TCO}
\Phi(t)<\infty \ \ \forall\,t>0.
\end{equation}
The condition
\begin{enumerate}
\item[{4}.]
{\ }
\vspace*{-4.0ex}
\begin{equation}
\label{MTS}
\varliminf_{|x|\to\infty}\frac{\ln(V_0(x)+V_1(x))}{\ln|x|}>0
\end{equation}
\end{enumerate}
is more than sufficient for \eqref{TCO}.\\

\noindent
\textbf{Conjecture}. Let \(V_0(x)\) and \(V_1(x)\) be functions from \(C(\mathbb{R})\)
which satisfy the conditions \eqref{PoC},  \eqref{UbC}, and  \eqref{MTS}. Let \(\Phi(t)\)
be the "trace" function constructed from the eigenvalues \(\lambda_n(t)\) of the eigenvalue
problem \eqref{GEVPr}. Then the function \(\Phi(-t)\) is absolutely monotonic on
\((-\infty,0)\), so the function \(\Phi\) admits the integral representation
\begin{equation}
\label{IRTF}
\Phi(t)=\int\limits_{\lambda\in[0,+\infty)}e^{-t\lambda}\,d\sigma(\lambda),\ \ t\in(0,+\infty),
\end{equation}
where \(d\sigma(\lambda)\) is a non-negative measure supported on \([0,+\infty)\).

If both functions \(V_0\) and \(V_1\) are even, then each of the "partial trace functions"
\(\Phi_{\textup{ev}}\) and \(\Phi_{\textup{od}}\) admit the integral representation
\begin{equation}
\label{IRPTF}
\Phi_{\textup{ev}}(t)=\hspace*{-8pt}\int\limits_{\lambda\in[0,+\infty)}
\hspace*{-8pt}e^{-t\lambda}\,
d\sigma_{\textup{ev}}(\lambda), \ \
\Phi_{\textup{od}}(t)=\hspace*{-8pt}\int\limits_{\lambda\in[0,+\infty)}\hspace*{-8pt}
e^{-t\lambda}\,
d\sigma_{\textup{od}}(\lambda), \ \ t\in(0,\infty),
\end{equation}
where \(d\sigma_{\textup{ev}}\) and \(d\sigma_{\textup{od}}\) are non-negative
measures supported on \([0,\infty)\).\\

It is interesting to confirm this conjecture even in the special case  of the one parametric family of quartic oscillators
\begin{equation}
\label{FQO}
-\frac{d^2}{dx^2}+(x^2+tx^4).
\end{equation}

\noindent
\textbf{Question.} For which functions \(V_0\) and \(V_1\), each summand \(e^{-\lambda_n(t)}\) of the "trace sum" \eqref{GeTr} possesses the property "\emph{the function \(e^{-\lambda_n(-t)}\)
is absolutely monotonic on \((-\infty,0)\)}".


\begin{thebibliography}{1}
\bibitem{BMV} D.\,Bessis, P.\,Moussa,\,M.\,Villani. \textit{Monotonic converging variational approximations to the functional integrals in quantum statistical mechanics.} J. Mat. Phys.,
\textbf{16}:11 (1975), 2318 - 2325.
\bibitem{S1} H.\,Stahl. \textit{Proof of the BMV conjecture.} arXiv:1107.4875v1, 1-56, 25 Jul2011.
\bibitem{S2} H.\,Stahl. \textit{Proof of the BMV conjecture.} arXiv:1107.4875v3, 1-25, 17 Aug2012.
\bibitem{S3} H.\,Stahl. \textit{Proof of the BMV conjecture.} Acta Math.,
\textbf{211} (2013), 255-290.
\bibitem{E} A.\,Eremenko. \textit{Herbert Stahl's proof of the BMV conjecture.} arXiv:1312.6003.
\bibitem{Gl} И.М.\,Глазман. \textit{Прямые Методы Качественного Спектрального Анализа Син\-гу\-ляр\-ных Дифференциальных Операторов.} Физматгиз, Москва, 1963. (In Russian). English Translation:\\
I.M.\,Glazman. \textit{Direct Methods of Qualitative Spectral Analysis of Singular
Differential Operators.}
Israel Program of Scientific Translations, Jerusalem, 1965.
\bibitem{deWM} J.S.\,de Wet, F.\,Mandl. \textit{On the asymptotic distribution of eigenvalues.}
Proc. Royal Soc. of London, \textbf{A 200} (1950), 572-580.
\bibitem{Ber0} S.N.\,Bernstein. \textit{Sur la d\'efinition et les propri\'et\'es
des fonctions analytiques d'une variable r\'eelle.} Math. Ann. \textbf{75} (1914), 449 - 468.
\bibitem{Ber1} S.N.\,Bernstein. \textit{Sur les functions absolument monotones.} Acta Math. \textbf{52} (1928), 1 - 66. (In French). Russian translation in: \cite{Ber2}, 370 - 425.
 \bibitem{Ber2}   С.Н.\,Бернштейн. \textit{Собрание Сочинений. Том 1.} Издательство АН СССР,
 1952.
\bibitem{A} Н.И.\,Ахиезер. \textit{Классическая Проблема Моментов.} Физматгиз, Москва, 1965.
(In Russian). English Transl.:\\
N.I.\,Akhiezer. \textit{The Clasical Moment Problem.} Oliver and Boyd, Edinburgh and London, 1965.
\bibitem{Wid} D.V.\,Widder. \textit{Laplace Transform.} Princeton Univ. Press, Princeton N.J.,
1946.
\bibitem{D} P.A.M.\,Dirac. \textit{The Principles of Quantum Mechanics.} 1st Edition, Clarendon
Press, Oxford 1930.
\bibitem{G} D.J.\,Griffiths. \textit{Introduction to Quantum Mechanics.} 2nd Edition, Pearson,
Upper Saddle River, N.J. 2005. x+468 pp.
\end{thebibliography}
\end{document}